   \documentclass[twoside,reqno,11pt]{fcaa-var} %

\usepackage{graphicx}
\usepackage{epsfig}
\usepackage{amsthm}
\usepackage{amsmath}
\usepackage{latexsym}
\usepackage{amsfonts}
\usepackage{amssymb}

\newtheoremstyle{theorem}
  {15pt}          
  {15pt}  
  {\sl}  
  {\parindent}
  {\sc}  
  {. }   
  { }    
  {}     
\theoremstyle{theorem}
\newtheorem{lemma}{Lemma}[section]
\newtheorem{theorem}{Theorem}[section]
\newtheorem{corollary}{Corollary}[section]

\newtheoremstyle{defi}
  {15pt}          
  {15pt}  
  {\rm}  
  {\parindent}     
  {\sc}  
  {. }    
  { }    
  {}     
\theoremstyle{defi}
\newtheorem{definition}{Definition}[section]
\newtheorem{remark}{Remark}[section]

 \def\proofname{\indent\rm P\ r\ o\ o\ f.}
 
 \def\qed{\hfill$\Box$} 

 \usepackage{hyperref} 

\usepackage{comment}
\usepackage{bbm}
\usepackage{cleveref}
\newtheorem{property}{Property}[section]

\newcommand{\RN}[1]{%
  \textup{\uppercase\expandafter{\romannumeral#1}}%
  }

 \def\theequation{\arabic{section}.\arabic{equation}}

 \def\themycopyrightfootnote{\vspace*{3pt}
   \par  \noindent 
   \hfill  \vspace*{-36pt}
  }

 \title[Symmetric Decompositions of $f\in L^2(\mathbb{R}) \cdots$ ]{Symmetric Decompositions of $f\in L^2(\mathbb{R})$ Via Fractional Riemann-Liouville Operators}
 \author[\normalsize  Y. Li]{\normalsize  Yulong Li$^1$}

\begin{document}

 \vbox to 2.5cm { \vfill }


 \bigskip \medskip

 \begin{abstract}
It is proved that

 given $-1/2<s<1/2$, for any $f\in L^2(\mathbb{R})$, there is a unique $u\in \widehat{H}^{|s|}(\mathbb{R})$ such that
$$
f=\boldsymbol{D}^{-s}u+\boldsymbol{D}^{s*}u\,,
$$
where $\boldsymbol{D}^{-s}, \boldsymbol{D}^{s*}$ are fractional Riemann-Liouville operators and the fractional derivatives are understood in the weak sense.  Furthermore, the regularity of $u$ is discussed, and other versions of the results are established. As an interesting consequence, the Fourier transform of elements of $L^2(\mathbb{R})$ is characterized.

 \medskip


 \smallskip

 \end{abstract}

 \maketitle

 \vspace*{-16pt}



\section{Introduction}
In this material, it is proved that every function $f\in L^2(\mathbb{R})$ could be written as a sum of fractional R-L integral and fractional R-L derivative of a certain function $u$ belonging to classical Sobolev space. The fractional R-L derivative is understood in the week sense which will be defined in section~\ref{section2}. Also other versions of this kind of decomposition will be pointed out. Those results will give author a start and a new point of view to study classical Sobolev space in the context of fractional calculus theory in the subsequent work. The material is organized as follows:
\begin{itemize}
\item Section~\ref{section:Notation} introduces the notations and conventions.
\item Section~\ref{section1} introduces the preliminary knowledge on fractional R-L operators.
\item Section~\ref{section2} introduces the characterization of $\widehat{H}^s(\mathbb{R})$ via R-L derivatives, which has been obtained in our previous work~\cite{V.Ging&Y.Li}.
\item Section~\ref{section3} establishes the main results.
\item Section~\ref{section4} proposes some interesting questions.
\end{itemize}
\section{Notations}\label{section:Notation}
Throughout the material, the following conventions are adopted:
\begin{itemize}
\item All the functions considered in this material are default to be real valued  unless otherwise specified.
\item $(f,g)$ and $ \int_\mathbb{R}fg$ shall be used interchangeably. Also, we denote integration $\int_A f$ on set $A$ without pointing out the variable unless it is necessary to specify.
\item $C_0^\infty(\mathbb{R})$ denotes the space of all infinitely differentiable functions with compact support in $\mathbb{R}$.
\item  $\mathcal{F}(u)$ denotes the Fourier transform of $u$ with specific expression defined in Definition~\ref{def:FT}, $ \widehat{u}$ denotes the Plancherel transform of $u$ defined in Theorem~\ref{thm:PAR}, which is well known that $ \widehat{u}$ is an isometry map from $L^2(\mathbb{R})$ onto $L^2(\mathbb{R})$ and coincides with $\mathcal{F}(u)$ if $u\in L^1(\mathbb{R})\cap L^2(\mathbb{R})$.
\item $u^\vee$ denotes  the inverse of Plancherel transform, and $*$ denotes convolution.
\end{itemize} 
\section{Preliminary}\label{section1} 
\subsection{Fractional Riemann-Liouville Integrals and Their Properties}
\begin{definition} \label{def:RLI}
Let $u:\mathbb{R} \rightarrow \mathbb{R}$ and  $\sigma >0$. The left and right Riemann-Liouville fractional integrals of order $\sigma$ are, formally respectively, defined as
\begin{align}
\boldsymbol{D}^{-\sigma}u(x)&:= \dfrac{1}{\Gamma(\sigma)}\int_{-\infty}^{x}(x-s)^{\sigma -1}u(s) \, {\rm d}s, \label{eq:LRLI}\\
\boldsymbol{D}^{-\sigma * }u(x) &:= \dfrac{1}{\Gamma(\sigma)}\int_{x}^{\infty}(s-x)^{\sigma-1}u(s) \, {\rm d}s,  \label{eq:RRLI}
\end{align}
where $\Gamma(\sigma)$ is the usual Gamma function. 
\end{definition}
\begin{property}[\cite{MR1347689}, p. 96]\label{prp:IntegrationExchange}
Given $0<\sigma$,
\begin{equation}
( \phi, \boldsymbol{D}^{-\sigma}\psi)=( \boldsymbol{D}^{-\sigma*}\phi, \psi),
\end{equation}
for $\phi\in L^p(\mathbb{R})$, $\psi \in L^q(\mathbb{R})$, $p>1, q>1$, $1/p+1/q=1+\sigma$.
\end{property}
\begin{property}[ \cite{MR1347689}, Theorem 7.1, p.138] \label{lem:FTFI} 
Assume that $u \in L^1(\mathbb{R})$ and $0<\sigma<1$, then
 \begin{equation}
  \mathcal{F}(\boldsymbol{D}^{-\sigma} u ) = (2\pi i\xi)^{-\sigma} \mathcal{F}(u) \text{ and }
  \mathcal{F}(\boldsymbol{D}^{-\sigma*}u )= (-2\pi i\xi)^{-\sigma} \mathcal{F}(u),\quad \xi \ne 0,
 \end{equation}
where $\mathcal{F}(\cdot)$ is the Fourier Transform  as defined in Definition~\ref{def:FT} .
 \end{property}

\begin{remark}\label{rem:ComplexPowerFunctions}
The complex power functions are understood as $(\mp i\xi)^{\sigma}=|\xi|^\sigma e^{\mp  \sigma \pi i \cdot \text{sign} (\xi)/2}$.
\end{remark}
\begin{property}[\cite{MR1347689}, pp. 95, 96]\label{pro:Translation}
Let $\mu>0$. Given $h\in \mathbb{R}$, define the translation operator $\tau_h$ as 
$\tau_h u(x) = u(x-h)$. Also, given $\kappa>0$, define the dilation operator
$\Pi_\kappa$ as $\Pi_\kappa u(x) = u(\kappa x)$.
Under the assumption that $\boldsymbol{D}^{-\mu}u$ and $\boldsymbol{D}^{-\mu*}u$ are well-defined, the following is true:
\begin{equation}
\begin{aligned}
\tau_h(\boldsymbol{D}^{-\mu}u)=\boldsymbol{D}^{-\mu}(\tau_hu)&,\quad \tau_h(\boldsymbol{D}^{-\mu*}u)=\boldsymbol{D}^{-\mu*}(\tau_hu)\\
\Pi_\kappa(\boldsymbol{D}^{-\mu}u)=\kappa^\mu \boldsymbol{D}^{-\mu}(\Pi_\kappa u)&,\quad
\Pi_\kappa(\boldsymbol{D}^{-\mu*}u)=\kappa^\mu \boldsymbol{D}^{-\mu*}(\Pi_\kappa u).
\end{aligned}
\end{equation}
\end{property}
\subsection{Fractional Riemann-Liouville Derivatives and Their Properties}
\begin{definition}\label{def:RLD}
Let $u:\mathbb{R} \rightarrow \mathbb{R}$. Assume $\mu >0$, $n$ is the smallest integer greater than $\mu$ (i.e., $n-1 \leq \mu < n$), and $\sigma = n- \mu$. The left and right Riemann-Liouville fractional derivatives of order $\mu$ are, formally respectively, defined as
\begin{align}
\boldsymbol{D}^{\mu} u &:= \dfrac{1}{\Gamma (\sigma)}\dfrac{{\rm d}^n}{{\rm d}x^n}\int_{-\infty}^{x}(x-s)^{\sigma-1}u(s) \, {\rm d}s, \label{5} \\
\boldsymbol{D}^{\mu*}u &:= \dfrac{(-1)^n}{\Gamma(\sigma)}\dfrac{{\rm d}^n}{{\rm d} x^n}\int_{x}^{\infty}(s-x)^{\sigma-1}u(s) \, {\rm d} s.\label{6}
\end{align}
\end{definition}
\begin{property}[\cite{V.Ging&Y.Li}] \label{prop:Boundedness}
Let $0<\mu $ and  $u\in C_0^\infty(\mathbb{R})$, then $\boldsymbol{D}^\mu u, \boldsymbol{D}^{\mu*} u \in L^p(\mathbb{R})$ for any $1\leq p<\infty$.
\end{property}
\begin{property}[ \cite{MR1347689}, p. 137] \label{lem:FTFD}
 Let $\mu > 0, u \in C_0^\infty(\mathbb{R})$, then
 \begin{equation} \label{eq:FTFDo}
 \mathcal{F}(\boldsymbol{D}^{\mu}u) = (2\pi i \xi)^{\mu} \mathcal{F}(u) \text{ and }
 \mathcal{F}(\boldsymbol{D}^{\mu*} u) = (-2\pi i \xi)^{\mu} \mathcal{F}(u), \quad \xi \ne 0, 
\end{equation}
where $\mathcal{F}(\cdot)$ is the Fourier Transform  as defined in Definition~\ref{def:FT} and as in Remark~\ref{lem:FTFI}, the complex power functions are understood as $(\mp i\xi)^{\sigma}=|\xi|^\sigma e^{\mp  \sigma \pi i \cdot \emph{sign} (\xi)/2}$.
\end{property}
\begin{property}[\cite{V.Ging&Y.Li}]\label{pro:TranslationDerivative} Consider $\tau_h$ and $\Pi_\kappa$ defined in Property~\ref{pro:Translation}. Let $\mu>0$, $n-1\leq \mu<n$, where $n$ is a positive integer, then
\begin{equation}
\begin{aligned}
\tau_h(\boldsymbol{D}^{\mu}u)=\boldsymbol{D}^{\mu}(\tau_hu)&,\quad \tau_h(\boldsymbol{D}^{\mu*}u)=\boldsymbol{D}^{\mu*}(\tau_hu)\\
\Pi_\kappa(\boldsymbol{D}^{\mu}u)=\kappa^{-\mu}\boldsymbol{D}^{\mu}(\Pi_\kappa u)&,\quad
\Pi_\kappa(\boldsymbol{D}^{\mu*}u)=\kappa^{-\mu}\boldsymbol{D}^{\mu*}(\Pi_\kappa u).
\end{aligned}
\end{equation}
\end{property}
Now we unify the notations by using $\boldsymbol{D}^\mu u$ and $\boldsymbol{D}^{\mu*} u$ for $\mu \in \mathbb{R}$. Namely, if $0\leq \mu$, they are understood as left and right fractional integrals, if $0<\mu$, as left and right derivatives. We adopt this convention throughout the rest of the material.
\section{Characterization of Sobolev Space $\widehat{H}^s(\mathbb{R})$}\label{section2}
In this section, we shall cite the results from our previous work~\cite{V.Ging&Y.Li}, which characterize the classical Sobolev space $\widehat{H}^s(\mathbb{R})$ defined in~\ref{thm:FTHsR}. This will  be convenient toward the main results in next section.

%
%
%
%
%
%
%
%
 %
 \begin{definition}[Weak Fractional Derivatives\cite{V.Ging&Y.Li}]\label{def:WFD}
 Let $\mu>0$, and $u, w \in L^1_{loc}(\mathbb{R})$.  The function $w$ is called weak
 $\mu$-order left fractional derivative of $u$, written as  $\boldsymbol{D}^\mu u = w$,
 provided
 \begin{equation}
(u, \boldsymbol{D}^{\mu* } \psi) = (w, \psi)~\forall \psi \in C_0^\infty(\mathbb{R}).
 \end{equation}
In a similar faschion, $w$ is weak $\mu$-order right fractional derivative of $u$, written as 
$\boldsymbol{D}^{\mu*} u=w$,  provided
 \begin{equation}
 (u, \boldsymbol{D}^{\mu}\psi) = (w, \psi) ~\forall \psi \in C_0^\infty(\mathbb{R}).
 \end{equation}
 \end{definition}
   \begin{definition}[\cite{V.Ging&Y.Li}]\label{def:FractionalSobolevSpaces}
 Let $s\geq 0$. Define spaces
 \begin{equation}
  \widetilde{W}^{s}_L(\mathbb{R})=\{u\in L^2(\mathbb{R}), \boldsymbol{D}^s u \in L^2(\mathbb{R})\}, 
 \end{equation}
 \begin{equation}
  \widetilde{W}^{s}_R(\mathbb{R})=\{u\in L^2(\mathbb{R}), \boldsymbol{D}^{s*} u\in L^2(\mathbb{R})\}, 
 \end{equation}
where $\boldsymbol{D}^s u$ and $\boldsymbol{D}^{s*} u$ are in the weak fractional derivative sense as defined in Definition~\ref{def:WFD}. A semi-norm
\begin{equation}
|u|_L:= \|\boldsymbol{D}^s u\|_{L^2(\mathbb{R})} ~~\text{for}~~\widetilde{W}^{s}_L(\mathbb{R}) ~~\text{and}~~ |u|_R:= \|\boldsymbol{D}^{s*} u\|_{L^2(\mathbb{R})}
~~\text{for}~~\widetilde{W}^{s}_R(\mathbb{R}),
\end{equation}
is given with the corresponding norm

\begin{equation}
\quad \|u\|_{\star}:=(\|u\|^2_{L^2(\mathbb{R})}+|u|_\star^2)^{1/2},~~\star=L,R.
\end{equation}
 \end{definition}
\begin{remark}\label{rem:Notations}
Notice, by convention,  $\widetilde{W}^{0}_L(\mathbb{R})=\widetilde{W}^{0}_R(\mathbb{R})
=L^2(\mathbb{R})=\widehat{H}^s(\mathbb{R})$. 
\end{remark}
Now we have the following characterization of  Sobolev space $\widehat{H}^s(\mathbb{R})$.
\begin{theorem}[\cite{V.Ging&Y.Li}]\label{thm:EquivalenceOfSpaces}
Given $s\geq 0$, $\widetilde{W}^{s}_L(\mathbb{R})$, $\widetilde{W}^{s}_R(\mathbb{R})$ and $\widehat{H}^s(\mathbb{R})$ are identical spaces with equal norms and semi-norms.
\end{theorem}
As a consequence, we have the following convenient result which will be of use in next section.
\begin{corollary}[\cite{V.Ging&Y.Li}]\label{cor:DensityOfC}
$u\in \widehat{H}^{s}(\mathbb{R})$ if and only if there exits a  sequence $\{u_n\}\subset C_0^\infty(\mathbb{R})$ such that $\{u_n\}, \{\boldsymbol{D}^su_n\}$ are Cauchy sequences in $L^2(\mathbb{R})$, with $\lim_{n\rightarrow \infty}u_n=u$. As a consequence, we have  $\lim_{n\rightarrow \infty}\boldsymbol{D}^{s}u_n=\boldsymbol{D}^{s}u$.
\\
Likewise,
\\
$u\in \widehat{H}^{s}(\mathbb{R}) $ if and only if there exits a  sequence $\{u_n\}\subset C_0^\infty(\mathbb{R})$ such that $\{u_n\}, \{\boldsymbol{D}^{s*}u_n\}$ are Cauchy sequences in $L^2(\mathbb{R})$, with $\lim_{n\rightarrow \infty}u_n=u$. As a consequence, we have $\lim_{n\rightarrow\infty}\boldsymbol{D}^{s*}u_n=\boldsymbol{D}^{s*}u$.
\end{corollary}
\section{Main Results} \label{section3}

In this section, under weak fractional derivative sense defined in Section~\ref{section2}, the following result will be established:
\begin{theorem}\label{MainTheorem1}
(1).
Given $ -1/2<s<1/2$, for $\forall$ $f\in L^2(\mathbb{R})$, there is a unique $u\in \widehat{H}^{|s|}(\mathbb{R})$ such that the following  decomposition holds:
\begin{equation}
f=\boldsymbol{D}^{-s}u+\boldsymbol{D}^{s*}u.
\end{equation}
Furthermore,  $f\in \widehat{H}^t(\mathbb{R})$ iff $u\in \widehat{H}^{s+t}(\mathbb{R})$, where $0<t$.
\\
(2). Given $ -1/2<s<1/2$, for any $f\in L^2(\mathbb{R})$, there exists a unique $u\in \widehat{H}^{|s|}(\mathbb{R})$ such that the following  decomposition holds:
\begin{equation}
f=\boldsymbol{D}^{-s}u+\boldsymbol{D}^{s}u.
\end{equation}
Furthermore,  $f\in \widehat{H}^t(\mathbb{R})$ iff $u\in \widehat{H}^{s+t}(\mathbb{R})$, where $0<t$.
\end{theorem}
\begin{remark}
It is worth noticing that, if $s$ is positive, $\boldsymbol{D}^{s*}u, \boldsymbol{D}^{s}u $  are understood as the weak fractional derivative of $u$ and $\boldsymbol{D}^{-s}u, \boldsymbol{D}^{-s*}u$ are understood in the usual sense, namely, the R-L integrals of $u$. Also, by similar arguments, we could also derive other variants or generalizations of Theorem~\ref{MainTheorem1}, such as:
\begin{equation}
 f=p\boldsymbol{D}^{-s}u+q\boldsymbol{D}^{s*}u, 
\end{equation}
where $p,q\in \mathbb{R}$ are suitable numbers (for example, $p>0, q>0$).

In the following, we focus only on Theorem~\ref{MainTheorem1}. 
\end{remark}
\subsection{Several Lemmas}
Toward the proof of Theorem~\ref{MainTheorem1} several necessary lemmas are first established. However, first we point out that the proof for the case $0<s<1/2$ and the case $-1/2<s\leq 0$ in Theorem~\ref{MainTheorem1} have no essential differences, simply because the sign change of $s$ results only in the exchange of notations of derivatives and integrals. For simplicity and without loss of generality, in the following we establish the proof of Theorem~\ref{MainTheorem1} only for the case $0<s<1/2$, and the proof for the case $-1/2<s\leq 0$ follows analogously without essential obstacle.
 \begin{lemma}\label{lem:dense}
Given $0<s<1/2$, then the set $M=\{w: w=\boldsymbol{D}^{-s}\psi+\boldsymbol{D}^{s*}\psi, ~\forall \psi\in C_0^\infty(\mathbb{R})$\} is dense in $L^{2}(\mathbb{R})$. Similarly,  the set
$\widetilde{M}=\{w: w=\boldsymbol{D}^{-s}\psi+\boldsymbol{D}^{s}\psi, ~\forall \psi\in C_0^\infty(\mathbb{R})$\} is dense in $L^{2}(\mathbb{R})$ also. 
\end{lemma}
\begin{proof}
The proof shall be established by invoking the theorem in \ref{Thm:dense}.
First we check the conditions for applying the theorem, and consider the set $M$.

 Since $0<s<1/2$, for $\psi \in C^\infty_0(\mathbb{R})$, it is true that $\boldsymbol{D}^{-s}\psi\in L^2(\mathbb{R})$ by applying Theorem 5.3  (\cite{MR1347689}, p. 103). By  Property~\ref{prop:Boundedness} we have $\boldsymbol{D}^{s*}\psi\in L^p(\mathbb{R}), p\geq 1$, thus  $M\subset L^2(\mathbb{R})$. Then it could be directly verified that $M$  is a subspace of $L^2(\mathbb{R})$ by checking closeness of addition and scalar multiplication.
 
Therefore all conditions are met to allow the utilization of \ref{Thm:dense}. Suppose now $g\in L^2(\mathbb{R})$ such that
$(g,w) = 0 ~\text{ for any } w\in M$.
The density of $M$ is confirmed if this last equation implies that $g = 0$.

Pick $0 \ne \varphi \in C_0^\infty(\mathbb{R})$, which is possible. Then  Plancherel Theorem in \ref{thm:PAR} gives $\widehat{\varphi}\neq 0$. On the account of continuity of $\widehat{\varphi}$, there exists a non-empty interval $(a,b)\subset \mathbb{R}$ such that $\widehat{\varphi}\neq 0$ on $(a,b)$.

Now set $v(x)=\varphi(\epsilon x)$,  with any fixed $\epsilon>0$. It is clear that $v \in C_0^\infty(\mathbb{R})$. 
 Using Property~\ref{lem:FTFI} and Property~\ref{lem:FTFD} (Fourier transform properties) for $w=\boldsymbol{D}^{-s}v+\boldsymbol{D}^{s*}v$ gives 
\begin{equation}\label{equ:WWW}
\widehat{w}(\xi)= \widehat{v}(\xi)\left((2\pi i \xi)^{-s}+(-2\pi i\xi)^s \right) = \dfrac{1}{\xi}\widehat{\psi}(\dfrac{\xi}{\epsilon})\left((2\pi i \xi)^{-s}+(-2\pi i\xi)^s \right).
\end{equation}
Following the practice stated in Remark~\ref{lem:FTFI}, it is easy to see that $(2\pi i \xi)^{-s}+(-2\pi i\xi)^s\neq 0$ a.e. on $\mathbb{R}$ by observing that $|(2\pi i \xi)^{-s}|\neq |(-2\pi i\xi)^s| $ a.e.. 
This implies $\widehat{w} \neq 0$ a.e. on $(\epsilon a, \epsilon b)$ since $\widehat{\psi}(\dfrac{\xi}{\epsilon})\neq 0$ a.e. on $(a\epsilon, b\epsilon)$. 
 
Now for $y\in \mathbb{R}$, we set the cross-correlation function
$$
G(-y) :=\int_\mathbb{R} g(x)w(x-y)\, {\rm d}x = \int_\mathbb{R} g(x)\, \tau_y w(x) \, {\rm d}x.
$$ 
Using Property~\ref{pro:Translation} and ~\ref{pro:TranslationDerivative} gives $\tau_y w\in M$ and thus $G(-y)=0$ for every $y\in \mathbb{R}$ by our assumption. Therefore by Plancherel Theorem $\widehat{G}\neq 0$. Notice  $G(y)=g(-x)*w(x)$. Since $g,w\in L^2(\mathbb{R})$, using convolution theorem (\cite{MR2218073}, Theorem 1.2, p.12) gives $\widehat{G}=\widehat{g(-x)}\cdot\widehat{w(x)}=0$.

Since  $\overline{\widehat{w}} \neq 0$ a.e. on $(\epsilon a, \epsilon b)$, it is concluded that  $\widehat{g(-x)}=0$ a.e.  on $(\epsilon a,\epsilon b)$.  Because $\epsilon>0$ is arbitrary, $\widehat{g(-x)}=0$ on any open interval, and thus $\widehat{g}=0$ on $\mathbb{R}$. Another use of Plancherel Theorem in \ref{thm:PAR} concludes that $g=0$ on $\mathbb{R}$, implying the density of $M$ in $L^2(\mathbb{R})$.

Finally, the same conclusion is true for $\widetilde{M}$ by repeating the similar foregoing calculation without essential difference.
 \end{proof}
 \begin{lemma}\label{singlenormestimate}
 Given $0<s<1/2, \psi\in C_0^\infty(\mathbb{R})$, then
 \begin{equation}
 (\boldsymbol{D}^{-s}\psi, \boldsymbol{D}^{s*}\psi)=\|\psi\|^2_{L^2(\mathbb{R})},\quad (\boldsymbol{D}^{-s}\psi, \boldsymbol{D}^{s}\psi)=\cos(s\pi)\|\psi\|^2_{L^2(\mathbb{R})}.
 \end{equation}
 \end{lemma}
 \begin{proof}
 Consider the first equality.
 
 From the proof of Lemma~\ref{lem:dense}, we already know, for $0<s<1/2$, $\boldsymbol{D}^{-s}\psi \in L^2(\mathbb{R})$ and $\boldsymbol{D}^{s*}\psi, \boldsymbol{D}^{s}\psi \in L^p(\mathbb{R}), p\geq 1$. This allows us to use Parseval Formula ~\ref{thm:ParsevalFormula} and Fourier Transform Properties~\ref{lem:FTFI}, ~\ref{lem:FTFD} to obtain
 \begin{equation}
 \begin{aligned}
 (\boldsymbol{D}^{-s}\psi, \boldsymbol{D}^{s*}\psi)&=(\widehat{\boldsymbol{D}^{-s}\psi}, \overline{\widehat{\boldsymbol{D}^{s*}\psi}})\\
 &=(\mathcal{F}(\boldsymbol{D}^{-s}\psi), \overline{\mathcal{F}(\boldsymbol{D}^{s*}\psi)})\\
 &=((2\pi i\xi)^{-s}\widehat{\psi}, \overline{(-2\pi i \xi)^{s}\widehat{\psi}}).
 \end{aligned}
 \end{equation}
 Invoking Remark~\ref{rem:ComplexPowerFunctions}, we find
 \begin{equation}
 ((2\pi i\xi)^{-s}\widehat{\psi}, \overline{(-2\pi i \xi)^{s}\widehat{\psi}})=\|\psi\|^2_{L^2(\mathbb{R})}.
 \end{equation}
 For the second equality, similarly we have
 \begin{equation}
 \begin{aligned}
 (\boldsymbol{D}^{-s}\psi, \boldsymbol{D}^{s}\psi)&=(\widehat{\boldsymbol{D}^{-s}\psi}, \overline{\widehat{\boldsymbol{D}^{s}\psi}})\\
 &=(\mathcal{F}(\boldsymbol{D}^{-s}\psi), \overline{\mathcal{F}(\boldsymbol{D}^{s}\psi)})\\
 &=((2\pi i\xi)^{-s}\widehat{\psi}, \overline{(2\pi i \xi)^{s}\widehat{\psi}}).
 \end{aligned}
 \end{equation}
 Again, by invoking Remark~\ref{rem:ComplexPowerFunctions}, this becomes
 \begin{equation}
 \begin{aligned}
 &((2\pi i\xi)^{-s}\widehat{\psi}, \overline{(2\pi i \xi)^{s}\widehat{\psi}})\\
 &=e^{is\pi}\int_{-\infty}^0  (-2\pi \xi)^{-s}\widehat{\psi}\cdot(-2\pi \xi)^s\overline{\widehat{\psi}}\, {\rm d}\xi +e^{-is\pi}\int_{-\infty}^0  (2\pi \xi)^{-s}\widehat{\psi}\cdot(2\pi \xi)^s\overline{\widehat{\psi}}\, {\rm d}\xi\\
 &= e^{is\pi}\int_{-\infty}^0 |\widehat{\psi}|^2\, {\rm d}\xi+e^{-is\pi}\int_{-\infty}^0  |\widehat{\psi}|^2\, {\rm d} \xi\\
 &=\cos(s\pi)\|\psi\|^2_{L^2(\mathbb{R})}.
 \end{aligned}
 \end{equation}
 In the last step we have used the fact that $\overline{\widehat{\psi}(-\xi)}=\widehat{\psi}(\xi)$ for real valued function $\psi$.
 \end{proof}
 \begin{lemma}\label{normestimate}
 Given $0<s<1/2$,  $\psi \in C^\infty_0(\mathbb{R})$, then
 \begin{align}
 \|\boldsymbol{D}^{-s}\psi+\boldsymbol{D}^{s*}\psi\|^2_{L^2(\mathbb{R})} &=\|\boldsymbol{D}^{-s}\psi\|^2_{L^2(\mathbb{R})}+\|\boldsymbol{D}^{s*}\psi\|^2_{L^2(\mathbb{R})}+2\|\psi\|^2_{L^2(\mathbb{R})},\label{euqation1}\\
 \|\boldsymbol{D}^{-s}\psi+\boldsymbol{D}^{s}\psi\|^2_{L^2(\mathbb{R})} &=\|\boldsymbol{D}^{-s}\psi\|^2_{L^2(\mathbb{R})}+\|\boldsymbol{D}^{s}\psi\|^2_{L^2(\mathbb{R})}+2\cos(s\pi)\|\psi\|^2_{L^2(\mathbb{R})}.\label{equation2}
 \end{align}

 \end{lemma}
 \begin{proof}
This is a direct consequence of  Lemma~\ref{singlenormestimate}.
 \end{proof}
 Now we are in the position to prove Theorem~\ref{MainTheorem1} for $0<s<1/2$.
 \begin{proof}
The proof is shown only for Case (1) in Theorem~\ref{MainTheorem1},  since Case (2) could be established analogously without obstacle. And it will be convenient to keep in mind the fact of Theorem~\ref{thm:EquivalenceOfSpaces} in the following.\\
Step 1.  

Given $0<s<1/2, f\in L^2(\mathbb{R})$. By Lemma~\ref{lem:dense}, there exists a sequence $\{\psi_n\}\subset C^\infty_0(\mathbb{R})$ such that
 \begin{equation}
f= \lim_{n\rightarrow \infty}(\boldsymbol{D}^{-s}\psi_n+\boldsymbol{D}^{s*}\psi_n)\, , \quad \text{  in $L^2(\mathbb{R})$}.
\end{equation} 
It is clear that $\{\boldsymbol{D}^{-s}\psi_n+\boldsymbol{D}^{s*}\psi_n\}$ is a Cauchy in $L^2(\mathbb{R})$. Equation~\eqref{euqation1} implies that $\{\boldsymbol{D}^{-s}\psi_n\}, \{\boldsymbol{D}^{s*}\psi_n\}$ and $\{\psi_n\}$ are Cauchy sequences separately since each term on the right hand side has a positive coefficient, namely, $1, 1, 2$. 
 
 Denote limit function $u=\lim_{n\rightarrow \infty}\psi_n$ and by invoking Corollary~\ref{cor:DensityOfC}, we know $u\in \widehat{H}^s(\mathbb{R})$ and $\lim_{n\rightarrow \infty}\boldsymbol{D}^{s*}\psi_n=\boldsymbol{D}^{s*}u$ in the weak fractional derivative sense. 

Then denote limit function  $v=\lim_{n\rightarrow \infty}\boldsymbol{D}^{-s}\psi_n$, and we claim $v=\boldsymbol{D}^{-s}u$, where $\boldsymbol{D}^{-s}u$ is to be understood in usual sense, namely, R-L integral of $u$ (which is well-defined). To see this,  the condition $0<s<1/2$ allows to use Property~\ref{prp:IntegrationExchange} to obtain
\begin{equation}\label{equ:Weakfractional}
(\boldsymbol{D}^{-s}u,\phi)=(u,\boldsymbol{D}^{-s*}\phi),\quad \forall \phi\in C_0^\infty(\mathbb{R}).
\end{equation}
On the other hand,
\begin{equation}
(v,\phi)=\lim_{n\rightarrow \infty}(\boldsymbol{D}^{-s}\psi_n,\phi)=\lim_{n\rightarrow \infty}(\psi_n,\boldsymbol{D}^{-s*}\phi)=( u,\boldsymbol{D}^{-s*}\phi).
\end{equation}
Thus, $(\boldsymbol{D}^{-s}u-v,\phi)=0, \forall \phi \in C_0^\infty(\mathbb{R})$, which deduces $v=\boldsymbol{D}^{-s}u$ a.e.. And therefore, $f=\boldsymbol{D}^{-s}u+\boldsymbol{D}^{s*}u$.\\
Step 2.

The uniqueness of $u$ is from the norm estimate in Equation~\eqref{euqation1}. Suppose $f=\boldsymbol{D}^{-s}u_1+\boldsymbol{D}^{s*}u_1=\boldsymbol{D}^{-s}u_2+\boldsymbol{D}^{s*}u_2, u_1,u_2\in \widehat{H}^s(\mathbb{R})$, then it is deduced that $\|u_1-u_2\|^2_{L^2(\mathbb{R})}=0$ and thus $u_1=u_2$ in $L^2(\mathbb{R})$.\\
Step 3.

Now suppose $f\in \widehat{H}^t(\mathbb{R})$, where $0<t$, we intend to show  $u\in \widehat{H}^{s+t}(\mathbb{R})$. Let's for now suppose $t\leq 2s$. First note the fact that $\boldsymbol{D}^{-s}u\in \widehat{H}^{2s}(\mathbb{R})$. Actually, using Property~\ref{prp:IntegrationExchange}, which is permissible here, gives
\begin{equation}
(\boldsymbol{D}^{-s}u, \boldsymbol{D}^{2s*}\phi)=(u, \boldsymbol{D}^{-s*}\boldsymbol{D}^{2s*}\phi)=(u,\boldsymbol{D}^{s*}\phi)=(\boldsymbol{D}^{s}u, \phi)\,, \forall \phi \in C_0^\infty(\mathbb{R}).
\end{equation}
The last equality above was by $u\in \widetilde{W}^s_L(\mathbb{R})$.
Therefore, by definition, $\boldsymbol{D}^{-s}u\in \widetilde{W}^{2s}_L(\mathbb{R})$, namely $\boldsymbol{D}^{-s}u\in \widehat{H}^{2s}(\mathbb{R})$ by Theorem~\ref{thm:EquivalenceOfSpaces}.
Then $\boldsymbol{D}^{s*}u=f-\boldsymbol{D}^{-s}u \in\widetilde{W}^t_L(\mathbb{R})$ by noticing our assumption $t\leq 2s$ and the fact that $\widehat{H}^{2s}(\mathbb{R})\subset\widehat{H}^{t}(\mathbb{R})$. Another use of Theorem~\ref{thm:EquivalenceOfSpaces} concludes $\boldsymbol{D}^{s*}u \in \widetilde{W}^t_R(\mathbb{R})$, which implies by definition of weak fractional derivative that
\begin{equation}
(\boldsymbol{D}^t(\boldsymbol{D}^{s*}u), \phi)=(\boldsymbol{D}^{s*}u, \boldsymbol{D}^t\phi), \quad \forall \phi \in C_0^\infty(\mathbb{R}).
\end{equation}
Observe that 
\begin{equation}
(\boldsymbol{D}^{s*}u, \boldsymbol{D}^t\phi)=\lim_{n\rightarrow \infty}(\boldsymbol{D}^{s*}\psi_n, \boldsymbol{D}^t\phi)=\lim_{n\rightarrow \infty}(\psi_n, \boldsymbol{D}^{s+t}\phi), \quad \forall \phi\in C_0^\infty(\mathbb{R}),
\end{equation}
which concludes $u\in \widetilde{W}^{s+t}_R(\mathbb{R})$, namely $u\in \widehat{H}^{s+t}(\mathbb{R})$.
Thus we actually raise the regularity of $u$ to $ \widehat{H}^{s+t}(\mathbb{R})$ from $\widehat{H}^s(\mathbb{R})$. 

For $t>2s$, we just need to rewrite $t=2s+\boldsymbol{Residue}$, and repeat above procedure to raise the regularity of $u$ from $\widehat{H}^s(\mathbb{R})$ to $\widehat{H}^{2s}(\mathbb{R})$, and repeat the same procedure again for $\boldsymbol{Residue}$, all the way to $\widehat{H}^{s+t}(\mathbb{R})$. \\
Step 4.

Now suppose $u\in \widehat{H}^{s+t}(\mathbb{R})$, where $0<t$, we intend to show $f\in \widehat{H}^t(\mathbb{R})$. By definition, it is easy to verify that $f- \boldsymbol{D}^{-s}u=\boldsymbol{D}^{s*}u\in \widetilde{W}^t_R(\mathbb{R})$, and $\boldsymbol{D}^{-s}u \in \widetilde{W}^{2s+t}_L(\mathbb{R})$, therefore, $\boldsymbol{D}^{-s}u+\boldsymbol{D}^{s*}u=f \in \widehat{H}^t(\mathbb{R})$ by using Theorem~\ref{thm:EquivalenceOfSpaces} and the fact that $\widehat{H}^{t_1}(\mathbb{R})\subset\widehat{H}^{t_2}(\mathbb{R})$ for $t_1\geq t_2$.
This completes the  proof for the case $0<s<1/2$. 

As mentioned, the case $-1/2< s\leq 0$ could be established analogously which completes the whole proof of Theorem~\ref{MainTheorem1}.
 \end{proof}
Based on Theorem~\ref{MainTheorem1}, the Fourier transform of $f\in L^2(\mathbb{R})$ therefore could be characterized as follows:
 \begin{corollary}
(1).
 Given $f\in L^2(\mathbb{R})$, $-1/2<s<1/2$, then there is a unique $u\in \widehat{H}^s(\mathbb{R})$ such that
 \begin{equation} 
 \widehat{f}(\xi)=\left((2\pi i \xi)^{-s}+(-2\pi i \xi)^s\right) \widehat{u}(\xi).
 \end{equation}

 (2).
 Given $f\in L^2(\mathbb{R})$, $-1/2<s<1/2$, then there is a unique $u\in \widehat{H}^s(\mathbb{R})$ such that
 \begin{equation} 
 \widehat{f}(\xi)=\left((2\pi i \xi)^{-s}+(2\pi i \xi)^s\right) \widehat{u}(\xi).
 \end{equation}

 The complex power functions are understood as $(\mp i\xi)^{\sigma}=|\xi|^\sigma e^{\mp  \sigma \pi i \cdot \emph{sign} (\xi)/2}$.
 \end{corollary}
 \begin{proof}
 Again, the proof is shown only for the case $0<s<1/2$ and part (1), the case $-1/2<s\leq 0$ and part (2) could be established analogously without  essential differences.
 
 Fix $f\in L^2(\mathbb{R})$, from Theorem~\ref{MainTheorem1} and the proof, there is a unique $u\in \widehat{H}^s(\mathbb{R})$ and a Cauchy sequence $\{\psi_n\}\subset C_0^\infty(\mathbb{R})$ such that
 \begin{equation}
 f=\boldsymbol{D}^{-s}u+\boldsymbol{D}^{s*}u,
 \end{equation}
 and 
 \begin{equation}
 \lim_{n\rightarrow \infty} \boldsymbol{D}^{-s}\psi_n=\boldsymbol{D}^{-s}u, \quad \lim_{n\rightarrow \infty}\boldsymbol{D}^{s*}\psi_n=\boldsymbol{D}^{s*}u, \quad \lim_{n\rightarrow \infty}\psi_n=u,\quad \text{in $L^2(\mathbb{R})$.}
 \end{equation}
 \end{proof}
Then we know
\begin{equation}
\widehat{f}=\widehat{\boldsymbol{D}^{-s}u}+\widehat{\boldsymbol{D}^{s*}u}\,, \text{and} \, \{\widehat{\boldsymbol{D}^{-s}\psi_n}\}\,,\{\widehat{\boldsymbol{D}^{s*}\psi_n}\}\,,\{\widehat{\psi_n}\} \,\text{ converge in $L^2(\mathbb{R})$}.
\end{equation}
On one hand, there is a subsequence $\{\widehat{\psi_n}_{i}\}$ that converge pointwise almost everywhere to $\widehat{u}$ (\ref{the:ConvergesInLp}), therefore $(2\pi i \xi)^{-s}\widehat{\psi_n}_i$ converges pointwise to $(2\pi i \xi)^{-s}\widehat{u}$ a.e. and $(-2\pi i \xi)^{s}\widehat{\psi_n}_i$ converges pointwise to $(-2\pi i \xi)^{-s}\widehat{u}$ a.e.
On the other hand, in $L^2(\mathbb{R})$
\begin{equation}
\begin{aligned}
\widehat{\boldsymbol{D}^{-s}u}&=\lim_{n\rightarrow \infty}\widehat{\boldsymbol{D}^{-s}{\psi_n}_i}=\lim_{n\rightarrow \infty}(2\pi i \xi)^{-s}\widehat{{\psi_n}_i}\quad,\\
\widehat{\boldsymbol{D}^{s*}u}&=\lim_{n\rightarrow \infty}\widehat{\boldsymbol{D}^{s*}{\psi_n}_i}=\lim_{n\rightarrow \infty}(-2\pi i \xi)^{s}\widehat{{\psi_n}_i}\quad.
\end{aligned}
\end{equation}
 Therefore it is easy to see
 \begin{equation}
 \widehat{\boldsymbol{D}^{-s}u}=(2\pi i \xi)^{-s}\widehat{u}\, \,\text{a.e.},\quad \widehat{\boldsymbol{D}^{s*}u}=(-2\pi i \xi)^{-s}\widehat{u} \,\,\text{a.e..}
 \end{equation}
 And thus
 \begin{equation}
 \widehat{f}(\xi)=\left((2\pi i \xi)^{-s}+(-2\pi i \xi)^s\right) \widehat{u}(\xi).
 \end{equation}
 This completes the whole proof.
 %
 %
 \section{Conclusion}\label{section4}
 We have constructed a bunch of maps $T_s: f \mapsto u$ from $L^2(\mathbb{R})$ to $\widehat{H}^s(\mathbb{R})$ in Theorem~\ref{MainTheorem1}, which depends on $s$. Naturally interesting questions would be`` Is this map onto? what is the maximum of $|s|$?'', and furthermore,``what are the potential algebraic structures underlying these operators?"
\appendix
\renewcommand{\theequation}{\thesection.\arabic{equation}}
\numberwithin{equation}{section}
\section{Several Definitions and Theorems}

\begin{definition}[Sobolev Spaces Via Fourier Transform]\label{thm:FTHsR}
Let $\mu\geq 0$. Define
\begin{equation}
\widehat{H}^\mu(\mathbb{R}) = \left \{ w \in L^2(\mathbb{R}) : \int_{\mathbb{R}} (1 + |2\pi\xi|^{2\mu}) |\widehat{w}(\xi) |^2 \, {\rm d} \xi < \infty \right \},
\end{equation}
where $\widehat{w}$ is Plancherel  transform defined in Theorem~\ref{thm:PAR}.
The space is endowed with semi-morn 
\begin{equation}
|u|_{\widehat{H}^\mu(\mathbb{R})}:=\||2\pi\xi|^\mu \widehat{u}\|_{L^2(\mathbb{R})},
\end{equation}
and norm
\begin{equation}
\|u\|_{\widehat{H}^\mu (\mathbb{R})}:=\left(\|u\|^2_{L^2(\mathbb{R})} +|u|^2_{\widehat{H}^\mu(\mathbb{R})}\right)^{1/2}.
\end{equation}
\end{definition}
\noindent
And it is well-known that $\widehat{H}^\mu(\mathbb{R})$ is a Hilbert space. 
\begin{definition}[Fourier Transform] \label{def:FT}
Given a function $f: \mathbb{R} \rightarrow \mathbb{R}$, the Fourier Transform of $f$ is defined as
$$
\mathcal{F}(f)(\xi):=\int_{-\infty}^{\infty}e^{-2\pi i x\xi}f(x)~ {\rm d}x \quad \forall \xi \in \mathbb{R}.
$$
\end{definition}
\begin{theorem}[Plancherel Theorem ( \cite{MR924157} p. 187)] \label{thm:PAR}

One can associate to each $f\in L^2(\mathbb{R})$ a function $\widehat{f}\in L^2(\mathbb{R})$ so that the following properties hold:
 \begin{itemize}
 \item If $f\in L^1(\mathbb{R})\cap L^2(\mathbb{R})$, then $\widehat{f} $ is the defined Fourier transform of $f$ in Definition \ref{def:FT}.
 \item For every $f\in L^2(\mathbb{R})$, $\|f\|_2=\|\widehat{f}\|_2$.
 \item The mapping $f\rightarrow \widehat{f}$ is a Hilbert space isomorphism of $L^2(\mathbb{R})$ onto $L^2(\mathbb{R})$.
 \end{itemize}
\end{theorem}
\begin{theorem}(\cite{MR2597943}, p. 189 )\label{thm:ParsevalFormula}
Assume $u,v\in L^2(\mathbb{R}^n)$. Then
\begin{itemize}
\item $\displaystyle \int_{\mathbb{R}^n} u\overline{v}=\int_{\mathbb{R}^n} \widehat{u}\overline{\widehat{v}}.$
\item $u=(\widehat{u})^{\vee}.$
\end{itemize}
\end{theorem}
\begin{theorem}[\cite{MR2328004}, Theorem 4.3-2, p. 191]\label{Thm:dense}
Let $(X,(\cdot,\cdot))$ be a Hilbert space and let $Y$ be a subspace of $X$, then $\overline{Y}=X$ if and only if element $x\in X$ that satisfy $(x,y)=0$ for all $y\in Y$ is $x=0$.
\end{theorem}
%
%
%
%
%
\begin{theorem}(\cite{MR1157815})\label{the:ConvergesInLp}
If $1\leq p\leq \infty$ and if $f_n$ is a Cauchy sequence in $L^p(\mathbb{R})$ with limit $f$, then $\{f_n\}$ has a subsequence which converges pointwise almost everywhere to $f(x)$.
\end{theorem}




 \bigskip \smallskip

 \it

 \noindent
$^1$Department of Mathematics and Statistics \\
\phantom{$^1$}University of Wyoming \\
\phantom{$^1$}1000 E. University Avenue \\
\phantom{$^1$}Dept. 3036, Laramie, Wyoming, USA  \\
\phantom{$^1$}e-mail: liyulong0807101@gmail.com(yli25@uwyo.edu)\\
\phantom{$^1$}Received: June 27, 2018 \\[12pt]

\end{document}